\documentclass{amsart}

\usepackage{amssymb}
\usepackage{amsmath}
\usepackage{amsthm}
\usepackage[numbers,square,comma,compress]{natbib}
\usepackage[unicode,hypertex,final,pagebackref,colorlinks,bookmarksnumbered]{hyperref}
\hypersetup{
  pdftitle = A Method of the Study of the Cauchy Problem for a Singularly Perturbed Linear Inhomogeneous Differential Equation,
  pdfauthor = {Evgeny E. Bukzhalev, Alexey V. Ovchinnikov},
  pdfsubject = ordinary differential equations,
  pdfkeywords = {singular perturbations, Banach fixed-point theorem, method of asymptotic iterations, method of boundary functions, Routh\unichar{"2013}Hurwitz stability criterion}}

\allowdisplaybreaks

\newcommand{\CC}{\mathbb{C}}
\newcommand{\NN}{\mathbb{N}}
\newcommand{\RR}{\mathbb{R}}
\DeclareMathOperator{\re}{Re}

\newtheorem{lem}{Lemma}
\newtheorem{sta}{Proposition}
\newtheorem{con}{Corollary}
\theoremstyle{remark}
\newtheorem*{rem*}{Remark}

\begin{document}

\title[Cauchy Problem for a Singularly Perturbed Equation]
 {A Method of the Study \\
 of the Cauchy Problem\\
 for a Singularly Perturbed Linear \\
 Inhomogeneous Differential Equation}

\author{Evgeny~E.~Bukzhalev}
\address{E.~E.~Bukzhalev: M.~V.~Lomonosov Moscow State University,
 Moscow, Russia}
\email{bukzhalev@mail.ru}

\author{Alexey~V.~Ovchinnikov}
\address{A.~V.~Ovchinnikov: M.~V.~Lomonosov Moscow State University,
 Moscow, Russia;
 All-Russian Institute for Scientific and Technical Information
 of the Russian Academy of Sciences, Moscow, Russia}
\email{ovchinnikov@viniti.ru}

\date{\today}

\begin{abstract}
We construct a sequence that converges to a solution of the Cauchy
problem for a singularly perturbed linear inhomogeneous
differential equation of an arbitrary order. This sequence is also
an asymptotic sequence in the following sense: the deviation (in
the norm of the space of continuous functions) of its $n$th
element from the solution of the problem is proportional to the
$(n+1)$th power of the parameter of perturbation. This sequence
can be used for justification of asymptotics obtained by the
method of boundary functions.
\end{abstract}


\maketitle

\section{Introduction}

We propose an algorithm of construction of a sequence
\begin{equation*}
 \psi_n(x;\varepsilon)=(y^1_n(x;\varepsilon), \dots,
 y^m_n(x;\varepsilon))
\end{equation*}
that converges for each $\varepsilon \in (0,\varepsilon_0]$ with
respect to the norm of the space $C_m[0,X]$ of continuous
$m$-dimensional vector-valued functions of the argument $x \in
[0,X]$) to the function
\begin{equation*}
 \psi(x;\varepsilon)=\left(y(x;\varepsilon),
 \frac{d}{dx} y(x;\varepsilon), \dots,
 \frac{d^{m-1}}{{dx}^{m-1}} y(x;\varepsilon)\right),
\end{equation*}
where $y(x;\varepsilon)$ is a classical solution of the
problem~\eqref{de}--\eqref{ic}; for the value of $\varepsilon_0$
we obtain an explicit lower estimate. The construction and the
proof of convergence of the sequence $\psi_n(x;\varepsilon)$ are
based on the Banach fixed-point theorem for a contracting mapping
of a complete metric space (see~\cite{Khamsi_Kirk_2001_book}).
Since the contraction coefficient~$k$ of the mapping is a value of
order~$\varepsilon$ ($k<\varepsilon/\varepsilon_0$), so that the
deviation $y^i_n(x;\varepsilon)$ (with respect to the norm
of~$C[0,X]$) from $\dfrac{d^{i-1}}{{dx}^{i-1}} y(x;\varepsilon)$
is~$O(\varepsilon^{n+1})$ (for $0<\varepsilon \leq
\varepsilon_0$), we see that this result has also asymptotic
character.

Note that each successive element of the sequence
$\psi_n(x;\varepsilon)$ is the result of the action of a certain
operator on the previous element. Elements of such sequences are
usually called iterations and sequences themselves are said to be
iterative. In our case, iterations approach
to~$\psi(x;\varepsilon)$ (in the norm of~$C_m[0,X]$) sufficiently
rapidly; the rate of approach is asymptotically reciprocal
to~$\varepsilon$. Therefore, the algorithm of construction of the
sequence $\psi_n(x;\varepsilon)$ is a method of asymptotic
iterations (for detail, see~\cite{Boglaev_76_SMD, Boglaev_Zhdanov_Stel'makh_78_DE, Bukzhalev_2017_MUCMC, Bukzhalev_2017_CMMP}). The
sequences $y^i_n(x;\varepsilon)$ are also called asymptotic
iterative sequences of the $(i-1)$th derivative of the solution
$y(x;\varepsilon)$ of the problem considered.

The possibility of application of the method of asymptotic
iterations is related to the fulfillment of the condition
\eqref{aij} for coefficients of the right-hand side of the
equation. However, the fulfillment of these conditions allows one
to apply the method of boundary-layer functions (see, e.g.,
\cite{Vasil'eva_Butuzov_Kalachev_1995_book}). One can immediately
verify that the deviation $y^1_n(x;\varepsilon)$ from the $n$th
partial sum $Y_n(x;\varepsilon)$ (which is called the asymptotics
or the asymptotic expansion of $n$th order) of the series
$Y(x;\varepsilon)$ obtained by the method of boundary-layer
functions has the form~$O(\varepsilon^{n+1})$. Thus, the
convergence of the sequence $y^1_n(x;\varepsilon)$ enables the
using of the method of asymptotic iterations for the justification
of asymptotic expansions obtained by the method of boundary-layer
functions (i.e., to the proof of the fact that the difference of
$Y_n(x;\varepsilon)$ and the solution $y(x;\varepsilon)$ has the
form $O(\varepsilon^{n+1})$ uniformly with respect to $x \in
[0,X]$).

Note that the convergence (uniform with respect to $\varepsilon$)
as $\varepsilon \in (0,\varepsilon_0]$ of asymptotic sequences
$y^i_n(x;\varepsilon)$ is a fundamental advantage of the method of
asymptotic iterations over the method of boundary-layer functions,
which allows one to construct an asymptotic series, which is, in
general does not converge even for arbitrarily
small~$\varepsilon$. The reason is that the estimate of the
deviation of $y^1_n(x;\varepsilon)$ from $Y_n(x;\varepsilon)$,
which has the form~$O(\varepsilon^{n+1})$, is not uniform with
respect to~$n$, so that this deviation may be not infinitesimal as
$n\to\infty$ but even unboundedly increasing.

Another advantage of the sequence $\psi_n(x;\varepsilon)$ is the
possibility of construction of all its terms under modest
smoothness conditions for the functions $a_i(x)$ and $b(x)$: for
the construction of all $\psi_n(x;\varepsilon)$ it suffices that
$a_i(x), b(x) \in C^1[0,X]$, while for the construction of all
terms of the series $Y(x;\varepsilon)$ the infinite
differentiability of $a_i(x)$ and $b(x)$ is required.

\section{Statement of the Problem and Auxiliary Estimates}

Consider the Cauchy problem for the linear, inhomogeneous,
singularly perturbed differential equation of order~$m$:
\begin{gather}
 \label{de}
 \varepsilon^m  y^{(m)}
=\varepsilon^{m-1} a_{m-1}(x) y^{(m-1)} +\dots+a_0(x) y+b(x),
 \quad x \in (0,X];
 \\
 \label{ic}
 y(0;\varepsilon)=y^0,\quad \ldots,\quad
 y^{(m-1)}(0;\varepsilon)=\frac{y^{m-1}}{\varepsilon^{m-1}},
\end{gather}
where $\varepsilon>0$ is the perturbation parameter, $X>0$, $y^0, \ldots$,
$y^{m-1} \in \RR$, and $a_0(x), \ldots$, $a_{m-1}(x)$, $b(x) \in C^1[0,X]$.
Moreover, we assume that the coefficients $a_i(x)$ satisfy the Routh--Hurwitz
condition for all $x \in [0,X]$ (see, e.g., \cite{Gantmacher_2000_book}):
\begin{equation}\label{aij}
 \begin{gathered}
 -a_{00}(x)>0,
 \quad
 \begin{vmatrix}
 a_{00}(x) & a_{01}(x) \\
 a_{10}(x) & a_{11}(x) \\
 \end{vmatrix}>0,
 \quad \ldots,
 \\
 {(-1)}^m \begin{vmatrix}
 a_{00}(x) & \ldots & a_{0(m-1)}(x) \\
 \vdots & \ddots & \vdots \\
 a_{(m-1)0}(x) & \ldots & a_{(m-1)(m-1)}(x) \\
 \end{vmatrix}>0,
 \end{gathered}
\end{equation}
where
\begin{gather*}
 a_{ij}(x) := \begin{cases}
 a_{2i-j}(x) & \text{for $0 \leq 2i-j<m$},
 \\
 -1 & \text{for $2i-j=m$},
 \\
 0, & \text{for $2i-j<0$ or $2i-j>m$}.
 \end{cases}
\end{gather*}

Recall that for the fulfillment of the conditions~\eqref{aij} it
is necessary (and for $m \in \{1,2\}$ is also sufficiently) that
all $a_i(x)$ be negative.

Let $p$ be that mapping, which to each $x \in [0,X]$ puts in
corresponding the polynomial
\begin{gather}\label{p(x)}
 p(x) :=\lambda^m-a_{m-1}(x) \lambda^{m-1}
 -\cdots-a_1(x) \lambda-a_0(x).
\end{gather}
Since the degree of the polynomial $p(x)$ is~$m$ on the whole
segment~$[0,X]$, there exist functions $\lambda_1, \dots,
\lambda_m:[0,X] \to \CC$ such that
\begin{equation*}
 p(x)=(\lambda-\lambda_1(x)) \ldots (\lambda-\lambda_m(x))
\end{equation*}
for each $x \in [0,X]$; the functions $\lambda_1(x), \dots, \lambda_m(x)$ are
called roots of the polynomial $p(x)$. The ordered set $(\lambda_1, \dots,
\lambda_m)$ of the function $\lambda_i$ is called the vector-function of
roots of the mapping~$p$. Note that there exist infinitely many
vector-functions of roots since for each $x \in [0,X]$ we can list the roots
of the polynomial~$p(x)$ in various orders. We fix one of the possible
orderings.

By the Routh--Hurwitz criterion (see~\cite{Gantmacher_2000_book}), the real parts of the
roots of the polynomial $p(x)$ are negative if and only if its coefficients
$a_i(x)$ satisfy the inequalities~\eqref{aij}. Thus, for all $(i,x) \in \{1,
\dots, m\} \times [0,X]$, the inequality
\begin{equation}\label{re la}
 \re \lambda_i(x)<0
\end{equation}
holds.

We prove that each of the function $\re \lambda_i$ is bounded on
the segment $[0,X]$ from the above by a certain negative constant.

Let $P$ be the mapping that to each $M = (a_0, \dots, a_{m-1}) \in
\CC^m$ puts in corresponding the polynomial
\begin{gather}\label{PM}
 P(M) :=\lambda^m-a_{m-1} \lambda^{m-1}-\cdots-a_1 \lambda-a_0.
\end{gather}
Denote by $\{\Lambda\}$ the set of all mappings $\Lambda: \CC^m
\to \CC^m$, which to each $M \in \CC^m$ put in correspondence an
ordered set $(\lambda^1, \dots, \lambda^m)$ of roots of the
equation $P(M)=0$ (we assume that each root is repeated as many
times as its multiplicity). In fact, the choice of $\Lambda \in
\{\Lambda\}$ means the choice of numbering of roots of the
polynomial $P(M)$ for each $M \in \CC^m$. It is easy to verify
that for $m \geq 2$ the set $\{\Lambda\}$ contains no mappings
continuous in the whole space~$\CC^m$ (see~\cite{2017arXiv171000640B}). However, it is known that
for each $m$ and any point~$M_0 \in \CC^m$, there exists a mapping
$\Lambda_{M_0} \in \{\Lambda\}$ continuous at this point (see,
e.g., \cite{Ostrowski_1966_book}).

Let $\varphi$ be the mapping, which to each $\Lambda \in
\{\Lambda\}$ puts in correspondence the vector-function
$\lambda=(\lambda^1, \dots, \lambda^m) \in \big(\CC^m \to \CC
\big)^m$ whose components $\lambda^i$ to each $M \in \CC^m$ put in
correspondence the $i$th coordinates of~$\Lambda(M)$:
$\Lambda(M)=(\lambda^1(M), \dots, \lambda^m(M))$. Obviously,
$\varphi$ is a bijective correspondence between $\{\Lambda\}$ and
$\{\lambda\} :=\varphi(\{\Lambda\})$. Moreover, the continuity of
the mapping $\Lambda$ is equivalent to the continuity of the
corresponding vector-function $\varphi(\Lambda)$, which, in its
turn, is equivalent to the continuity of all its components.

\begin{lem}\label{lem1}
Let $\lambda=(\lambda^1, \dots, \lambda^m) \in \{\lambda\}$. Then
$$
 \overline \Lambda(M) :=\max \{\re\lambda^1(M), \dots,
 \re\lambda^m(M)\}
$$
is a continuous function of~$M$.
\end{lem}

\begin{rem*}
For each point $M \in \CC^m$, the unordered set of roots of the
polynomial $P(M)$ and the value $\overline \Lambda(M)$ are
independent of the choice of $\lambda \in \{\lambda\}$. Thus, to
each $\lambda \in \{\lambda\}$ (i.e., to each way of numbering of
roots of the polynomial~$P(M)$) the same function $\overline
\Lambda$ corresponds.
\end{rem*}

\begin{proof}[Proof of Lemma \ref{lem1}]
Fix an arbitrary point $M_0 \in \CC^m$ and choose a mapping
$\lambda_{M_0}=(\lambda^1_{M_0}$, $\dots,$ $\lambda^m_{M_0}) \in
\{\lambda\}$ continuous at this point. Each of the functions
$\lambda^i_{M_0}$ is also continuous at the point~$M_0$. But the
continuity of $\lambda^i_{M_0}$ implies the continuity of $\re
\lambda^i_{M_0}$, whereas the continuity of all $\re
\lambda^i_{M_0}$, in its turn, implies the continuity of the
maximum of these functions.
\end{proof}

\begin{con}\label{cor1}
There exist positive $\chi$ (independent of $i$ and $x$) such
that
\begin{equation*}
 \re\lambda_i(x)<-\chi
\end{equation*}
for all $(i,x)\in\{1,\dots,m\}\times[0,X]$, where $\lambda_i(x)$
is the $i$th root of the polynomial $p(x)$ (see \eqref{p(x)}) for
each $x \in [0,X]$.
\end{con}

\begin{rem*}
For each $x \in [0,X]$, the unordered set of roots of the
polynomial $p(x)$ and the value $\overline \lambda(x) :=\max
\{\re\lambda_1(x), \dots, \re\lambda_m(x)\}$ are independent of
the way of numbering of these roots.
\end{rem*}

\begin{proof}[Proof of Corollary \ref{cor1}]
Let $\lambda=(\lambda^1, \dots, \lambda^m)$ be an arbitrary
mapping from~$\{\lambda\}$. By the remark above, without loss of
generality, we can assume that
$$
 \lambda_i(x)=\lambda^i(a_0(x),\dots,a_{m-1}(x))
 \quad \forall (i,x) \in \{1, \dots, m\} \times [0,X].
$$
Since the function $\overline \lambda(x)$, which is equal to~$\overline
\Lambda(a_0(x),\dots,a_{m-1}(x))$, is continuous (as a composite function)
and negative (see~\eqref{re la}) on the whole segment $[0,X]$, by the
Weierstrass extreme-value theorem, there exists $x_0 \in [0,X]$ such that
\begin{equation}\label{vkp}
\begin{aligned}
 \chi :&=-\overline\lambda(x_0)
 =-\max_{[0,X]}\overline\Lambda\big(a_0(x),\dots,a_{m-1}(x)\big)
 \\
 &=-\max_{[0,X]}\max\big\{\re\lambda_1(x),\dots,
 \re\lambda_m(x)\big\}>0.
\end{aligned}
\end{equation}
The proof is complete.
\end{proof}

\begin{rem*}
One can prove that there exist continuous functions
$\lambda_1(x),\dots,\lambda_m(x)$ that describe the set of all
roots (with account of multiplicities) of the polynomial~$p(x)$
for each $x \in [0,X]$; here the fact that the variable~$x$ is
one-dimensional is substantial.
\end{rem*}

Consider the following auxiliary problem:
\begin{gather}
 \label{ae+}
 a_0(x) \bar y+b(x)=0, \quad x \in [0,X];
 \\
 \label{de+}
 \frac{d^{m}\Pi}{{d\xi}^{m}}
 =a_{m-1}(0) \frac{d^{m-1}\Pi}{{d\xi}^{m-1}}
 +\cdots+a_0(0) \Pi,
 \quad \xi \in \left(0,\frac{X}{\varepsilon}\right];
 \\
 \label{ic+}
 \Pi(0)=y^0-\bar y(0), \quad
 \frac{d\Pi}{d\xi}(0)=y^1, \quad \ldots, \quad
 \frac{d^{m-1}\Pi}{{d\xi}^{m-1}}(0)=y^{m-1}.
\end{gather}
Equation~\eqref{ae+} is an algebraic equation of the first degree
with respect to~$\bar y(x)$, whereas \eqref{de+} is an autonomous
homogeneous linear differential equation for~$\Pi(\xi)$. The
solution of the problem~\eqref{ae+}--\eqref{ic+} has the form
\begin{equation}\label{wty}
 \begin{aligned}
 \bar y(x)&=-\frac{b(x)}{a_0(x)},
 \\
 \Pi(\xi)&=\alpha_{11} e^{\lambda_1(0) \xi}+\cdots
 +\alpha_{1m_1} \xi^{m_1-1} e^{\lambda_{m_1}(0) \xi}+\cdots
 \\
 &+\alpha_{q1} e^{\lambda_{m_1+\cdots+m_{q-1}+1}(0) \xi}+\cdots
 +\alpha_{qm_q} \xi^{m_q-1}
 e^{\lambda_{m_1+\cdots+m_{q-1}+m_q}(0) \xi},
 \end{aligned}
\end{equation}
where $\lambda_1(0)=\dots=\lambda_{m_1}(0)$, \dots,
$\lambda_{m_1+\cdots+m_{q-1}+1}(0)=\dots=\lambda_{m_1+\cdots+m_q}(0)$
are roots of the polynomial~$p(0)$ (see~\eqref{p(x)}),
$\alpha_{11}$, \dots, $\alpha_{qm_q}$ are constants that are
uniquely expressed through $y^0-\bar y(0)$, $y^1$, \dots,
$y^{m-1}$ and $\lambda_1(0)$, \dots, $\lambda_m(0)$ (here
$m_1+\dots+m_q=m$).

We see from \eqref{wty} and \eqref{vkp} that for sufficiently
large $\tilde C$ the functions $\Pi^{(i)}(\xi)$ satisfy the
estimate
\begin{gather}\label{est wty}
 \big|\Pi^{(i)}(\xi)\big| \leq
 \tilde C (1+\xi^{m-1}) e^{-\chi \xi},
 \quad (i,\xi) \in \{0, \dots, m-1\} \times [0,+\infty).
\end{gather}

In the problem \eqref{de}--\eqref{ic}, we perform the following
change of variables:
\begin{equation}\label{subs}
 \begin{aligned}
 x&=\varepsilon \xi,
 \\
 y(x;\varepsilon)&=\tilde y(\xi,x)
 +\varepsilon z^1(\xi;\varepsilon),
 \\
 \frac{d^{i-1}y}{{dx}^{i-1}}(x;\varepsilon)
 &=\varepsilon^{1-i} \frac{d^{i-1}\Pi}{{d\xi}^{i-1}}(\xi)
 +\varepsilon^{2-i} z^i(\xi;\varepsilon),
 \quad i=\overline{2,m},
 \end{aligned}
\end{equation}
where $\tilde y(\xi,x) :=\bar y(x)+\Pi(\xi)$.

For the new functions $z^i(\xi;\varepsilon)$ we obtain the
following initial-value problem:
\begin{gather}
 \label{dez1}
 \frac{dz^1}{d\xi}=z^2-\bar y'(\varepsilon \xi), \quad
 \xi \in \left(0,\frac{X}{\varepsilon}\right];
 \\
 \label{de zi}
 \frac{dz^i}{d\xi}=z^{i+1}, \quad
 (i,\xi) \in \{2, \dots, m-1\}
 \times \left(0,\frac{X}{\varepsilon}\right];
 \\
 \label{de zm}
 \frac{dz^m}{d\xi}=a_{m-1}(\varepsilon \xi) z^m
 +\dots+a_0(\varepsilon \xi) z^1+f(\xi;\varepsilon), \quad
 \xi \in \left(0,\frac{X}{\varepsilon}\right];
 \\
 \label{iczi}
 z^1(0;\varepsilon)=\ldots=z^m(0;\varepsilon)=0
\end{gather}
(\eqref{dez1} only for $m \geq 2$, \eqref{de zi} only for $m \geq
3$), where
\begin{gather}
 \label{f}
 f(\xi;\varepsilon) := \begin{cases}
 \varepsilon^{-1} \Big\{ \big[a_{m-1}(\varepsilon \xi)-a_{m-1}(0)\big]
 \\
 \qquad \times \Pi^{(m-1)}(\xi)+\dots+\big[a_0(\varepsilon \xi)-a_0(0)\big]
 \Pi(\xi) \Big\} & \text{for $m \geq 2$};
 \\
 \varepsilon^{-1} \big[a_0(\varepsilon \xi)-a_0(0)\big]\Pi(\xi)
 -\bar y'(\varepsilon \xi) & \text{for $m=1$}.
 \end{cases}
\end{gather}

We transform Eq.~\eqref{de zm} adding the variable~$x$ as a new
parameter:
\begin{multline}\label{de zm+}
 \frac{dz^m}{d\xi}=a_{m-1}(x) z^m+\dots+a_0(x) z^1
 +\big[a_{m-1}(\varepsilon \xi)-a_{m-1}(x)\big] z^m+\dots
 \\
 +\big[a_0(\varepsilon \xi)-a_0(x)\big] z^1
 +f(\xi;\varepsilon),
 \quad (\xi,x) \in \left(0,\frac{X}{\varepsilon}\right] \times [0,X].
\end{multline}

The problem \eqref{dez1}, \eqref{de zi}, \eqref{de zm+},
\eqref{iczi} is equivalent to the following system of integral
equations:
\begin{multline}\label{ie}
 z^i(\xi;\varepsilon)=- \int_0^\xi
 \Phi^1_{\xi^{i-1}}(\xi-\zeta; x)
 \bar y'(\varepsilon \zeta) d\zeta
 \\
 +\int_0^\xi \Phi^m_{\xi^{i-1}}(\xi-\zeta; x)
 \Big\{\big[a_{m-1}(\varepsilon \zeta)-a_{m-1}(x)\big]
 z^m(\zeta;\varepsilon)
 \\
 +\dots+\big[a_0(\varepsilon \zeta)-a_0(x)\big]
 z^1(\zeta;\varepsilon)+f(\zeta;\varepsilon) \Big\} d\zeta,
 \\
 \quad (i,\xi,x) \in \overline{1,m}
 \times \left[0,\frac{X}{\varepsilon}\right] \times [0,X],
\end{multline}
where $\Phi^j_{\xi^{i-1}}(\xi-\zeta; x)=K^i_j(\xi,\zeta;x)$ are
the entries of the Cauchy matrix
\begin{gather*}
 K(\xi,\zeta;x) := \begin{bmatrix}
 \Phi^1(\xi-\zeta; x) & \Phi^2(\xi-\zeta; x)
 & \dots & \Phi^m(\xi-\zeta; x)
 \\
 \Phi^1_\xi(\xi-\zeta; x) & \Phi^2_\xi(\xi-\zeta; x)
 & \dots & \Phi^m_\xi(\xi-\zeta; x)
 \\
 \vdots & \vdots & \ddots & \vdots
 \\
 \Phi^1_{\xi^{m-1}}(\xi-\zeta; x) &
 \Phi^2_{\xi^{m-1}}(\xi-\zeta; x) & \dots &
 \Phi^m_{\xi^{m-1}}(\xi-\zeta; x)
 \end{bmatrix}
\end{gather*}
of the corresponding homogeneous system
\begin{equation*}
 \frac{dz^1}{d\xi}=z^2, \quad \dots, \quad
 \frac{dz^{m-1}}{d\xi}=z^m, \quad
 \frac{dz^m}{d\xi}=a_{m-1}(x) z^m+\dots+a_0(x) z^1.
\end{equation*}

Note that the functions $\Phi^1(\xi;x)$ and $\Phi^m(\xi;x)$ used
in~\eqref{ie}, due to the definition of the Cauchy matrix, are the
solutions of the following initial-value problems:
\begin{gather}
 \label{de Phi1}
 \frac{d^{m}\Phi^1}{{d\xi}^{m}}
 =a_{m-1}(x) \frac{d^{m-1}\Phi^1}{{d\xi}^{m-1}}
 +\dots +a_0(x) \Phi^1,
 \quad (\xi,x) \in \RR \times [0,X];
 \\
 \label{ic Phi1}
 \Phi^1(0;x)=1, \quad \frac{d\Phi^1}{d\xi}(0;x)
 =\ldots=\frac{d^{m-1}\Phi^1}{{d\xi}^{m-1}}(0;x)=0,
 \quad x \in [0,X];
 \\
 \label{de Phim}
 \frac{d^{m}\Phi^m}{{d\xi}^{m}}
 =a_{m-1}(x) \frac{d^{m-1}\Phi^m}{{d\xi}^{m-1}}
 +\dots+a_0(x) \Phi^m,
 \quad (\xi,x) \in \RR \times [0,X];
 \\
 \label{ic Phim}
 \Phi^m(0;x)=\ldots
 =\frac{d^{m-2}\Phi^m}{{d\xi}^{m-2}}(0;x)=0, \quad
 \frac{d^{m-1}\Phi^m}{{d\xi}^{m-1}}(0;x)=1,
 \quad x \in [0,X].
\end{gather}

From \eqref{de Phi1}--\eqref{ic Phim} and the theorems on the
continuity and differentiability with respect to parameters of
solutions of initial-value problems we conclude that $\Phi^1(\xi;x)$,
$\Phi^m(\xi;x) \in C^{\infty,1}(\RR \times [0,X])$.

Since the solution $(z^1,\dots,z^m)$ of the system~\eqref{ie} is
clearly independent of~$x$, we can replace $x$ in~\eqref{ie} by an
arbitrary function~$\xi$ and $\varepsilon$ with values in $[0,X]$. Then, setting $x=\varepsilon \xi$, we
arrive at the following equations for $z^i(\xi;\varepsilon)$:
\begin{multline}\label{iezi}
 z^i(\xi;\varepsilon)=-\int_0^\xi \Phi^1_{\xi^{i-1}}
 (\xi-\zeta; \varepsilon \xi)
 \bar y'(\varepsilon \zeta) d\zeta
 +\int_0^\xi \Phi^m_{\xi^{i-1}}(\xi-\zeta; \varepsilon \xi)
 \\
 \times \Big\{\big[a_{m-1}(\varepsilon \zeta)
 -a_{m-1}(\varepsilon \xi)\big] z^m(\zeta;\varepsilon)+\dots
 +\big[a_0(\varepsilon \zeta)-a_0(\varepsilon \xi)\big]
 z^1(\zeta;\varepsilon)+f(\zeta;\varepsilon) \Big\} d\zeta
 \\
 =:\widehat A_i(\varepsilon)[z^1,\dots,z^m](\xi;\varepsilon),
 \quad (i,\xi) \in \overline{1,m}
 \times \left[0,\frac{X}{\varepsilon}\right],
\end{multline}
(the first integral only for $m \geq 2$) or briefly
\begin{multline}\label{ie+}
 \big(z^1(\xi;\varepsilon), \dots, z^m(\xi;\varepsilon)\big)
 \\
 =\big(\widehat A_1(\varepsilon)[z^1, \dots, z^m]
 (\xi;\varepsilon),\ \dots,\
 \widehat A_m(\varepsilon)[z^1, \dots, z^m]
 (\xi;\varepsilon)\big)=
 \\
 =: \widehat A(\varepsilon)[z^1, \dots, z^m](\xi;\varepsilon),
 \quad \xi \in \left[0,\frac{X}{\varepsilon}\right],
\end{multline}
where for each fixed $\varepsilon \in (0,+\infty)$ by the domain of the
operator $\widehat A(\varepsilon)$ we mean the space $C_m[0,X/\varepsilon]$
of $m$-dimensional vector-functions continuous on the segment
$[0,X/\varepsilon]$:
\begin{equation*}
 \widehat A(\varepsilon) :
 C_m\left[0,\frac{X}{\varepsilon}\right]
 \to C_m\left[0,\frac{X}{\varepsilon}\right].
\end{equation*}

In the sequel we need one auxiliary property of the solution~$w$
of the Cauchy problem for a linear differential equation with
constant coefficients considered as parameters for~$w$:
\begin{gather}\label{dewn}
 \frac{d^{m}w}{{d\xi}^{m}}
 =a_{m-1} \frac{d^{m-1}w}{{d\xi}^{m-1}}+\cdots+a_0 w,
 \quad \xi \in (0,+\infty);
 \\
 \label{icwn}
 w(0;M_m,N_m)=w^0,\quad \ldots,\quad
 \frac{d^{m-1}w}{{d\xi}^{m-1}}(0;M_m,N_m)=w^{m-1},
\end{gather}
where $M_m=(a_0, \dots, a_{m-1}) \in \CC^m$ and $N_m=(w^0,
\dots, w^{m-1}) \in \CC^m$.

Introduce the following notation:
\begin{equation*}
 \overline \Lambda_m(M_m)
 :=\max \{\re\lambda^1(M_m), \dots, \re\lambda^m(M_m)\},
\end{equation*}
where $\lambda^1(M_m)$, \dots, $\lambda^m(M_m)$ are the roots of
the characteristic polynomial of Eq.~\eqref{dewn}
(see~\eqref{PM}),
\begin{equation*}
 \Pi_m(C) :=\big\{(x_1, \dots, x_m) \in \CC^m :
 |x_1| \leq C,\ \dots,\  |x_m| \leq C \big\}.
\end{equation*}

\begin{lem}\label{lem2}
Let $C_a \geq 0$ and $C_w \geq 0$. Then there exists $\tilde C_m
\geq 0$ such that
\begin{gather*}
 \left| \frac{d^{i}w}{{d\xi}^{i}}(\xi;M_m,N_m) \right|
 \leq \tilde C_m (1+\xi^{m-1})
 e^{\overline \Lambda_m(M_m) \xi}
\end{gather*}
for all $(i,\xi,M_m,N_m) \in \{0, \dots, m-1 \} \times [0,
+\infty) \times \Pi_m(C_a) \times \Pi_m(C_w)$, where
$w(\xi;M_m,N_m)$ is a solution of the
problem~\eqref{dewn}--\eqref{icwn}.
\end{lem}

The assertion of the theorem can be proved by induction on $m$.

\begin{con}
There exist $\chi>0$ and $C_\Phi>0$ such that
\begin{gather}\label{estPhi}
 |\Phi^1_{\xi^i}(\xi; x)|,\ |\Phi^m_{\xi^i}(\xi;x)|
 \leq C_\Phi (1+\xi^{m-1}) e^{-\chi \xi}
\end{gather}
for all $(i,\xi,x) \in \{0, \dots, m-1 \} \times [0,+\infty)
\times [0,X]$, where $\Phi^1(\xi;x)$ and $\Phi^m(\xi;x)$ are the
solutions of the problems \eqref{de Phi1}--\eqref{ic Phi1}
and~\eqref{de Phim}--\eqref{ic Phim}, respectively.
\end{con}

\begin{proof}
To prove the estimate \eqref{estPhi} it suffices to set
\begin{equation*}
 \chi :=- \max_{[0,X]}
 \max \big\{\re\lambda_1(x), \dots, \re\lambda_m(x)\big\}
\end{equation*}
(see \eqref{vkp}) and apply the Weierstrass extreme-value theorem
on the boundedness of a continuous function for $a_i(x)$ and
Lemma~\ref{lem2}.
\end{proof}

\section{Construction and Proof of Convergence \\
 of Iterative Sequence}

Let
\begin{gather*}
 O(\vartheta,C_0;\varepsilon) := \bigg\{ (z^1, \dots, z^m) \in C_m\left[0,\frac{X}{\varepsilon}\right] : \forall \xi \in \left[0,\frac{X}{\varepsilon}\right]
 \\
 (z^1(\xi), \dots, z^m(\xi)) \in {[-C_0,+C_0]}^m \bigg\}
\end{gather*}
be a closed $C_0$-neighborhood of the vector-function $(z^1, \dots,
z^m) \equiv (0, \dots, 0) =: \vartheta$ in the space~$C_m[0,X/\varepsilon]$.

\begin{sta}
There exist $\varepsilon_0>0$ and $C_0 \geq 0$ \textup{(}$C_0$ is
independent of~$\varepsilon$\textup{)} such that
\begin{equation*}
 \widehat A(C_0;\varepsilon) : O(\vartheta,C_0;\varepsilon)
 \to O(\vartheta,C_0;\varepsilon)
\end{equation*}
for any $\varepsilon \in (0,\varepsilon_0]$, where $\widehat
A(C_0;\varepsilon)=\big(\widehat A_1(C_0;\varepsilon), \dots,
\widehat A_m(C_0;\varepsilon) \big)$ is the restriction of the
operator $\widehat A(\varepsilon)$ to~$O(\vartheta,C_0;\varepsilon)$.
\end{sta}

\begin{proof}
We fix arbitrary $\varepsilon>0$ and $C_0 \geq 0$, apply the
operators $\widehat A_i(C_0;\varepsilon)$ to an arbitrary
vector-function $(z^1(\xi), \dots, z^m(\xi)) \in
O(\vartheta,C_0;\varepsilon)$ and, taking into account \eqref{iezi} and
\eqref{estPhi}, estimate the result obtained:
\begin{multline}\label{Ai}
 \left| \widehat A_i(C_0;\varepsilon)[z^1,\dots,z^m](\xi)\right|
 \\
 \leq C_\Phi e^{-\chi \xi}
 \Big\{C_0 \int_0^\xi e^{\chi \zeta}
 \big[1+{(\xi-\zeta)}^{m-1} \big]
 \big[|a_{m-1}(\varepsilon \zeta)-a_{m-1}(\varepsilon \xi)|+\dots
 \\
 +|a_0(\varepsilon \zeta)-a_0(\varepsilon \xi)| \big] d\zeta
 +\int_0^\xi e^{\chi \zeta}
 \big[1+{(\xi-\zeta)}^{m-1} \big]
 \big[|f(\zeta;\varepsilon)|+|\bar y'(\varepsilon \zeta)|
 \big] d\zeta \Big\}, \quad i=\overline{1,m}
\end{multline}
(the term $|\bar y'(\varepsilon \zeta)|$ only for $m \geq 2$).

For the first integral in~\eqref{Ai} we have
\begin{multline}\label{est int1}
 \int_0^\xi e^{\chi \zeta}
 \big[1+{(\xi-\zeta)}^{m-1} \big]
 \big[|a_{m-1}(\varepsilon \zeta)
 -a_{m-1}(\varepsilon \xi)|+\dots
 +|a_0(\varepsilon \zeta)-a_0(\varepsilon \xi)| \big] d\zeta
 \\
 =\varepsilon \int_0^\xi e^{\chi \zeta}
 \big[(\xi-\zeta)+{(\xi-\zeta)}^m \big] \Big\{
 \big|a_{m-1}'(\varepsilon [(1-\theta_{m-1}) \zeta
 +\theta_{m-1} \xi])\big|+\dots
 \\
 \shoveright{+\big| a_0'(\varepsilon [(1-\theta_0) \zeta+\theta_0 \xi])
 \big| \Big\} d\zeta}
 \\
 \shoveleft{\leq \varepsilon \big\{\| a_{m-1}'(x) \|+\dots
 +\| a_0'(x) \| \big\} \int_0^\xi e^{\chi \zeta}
 \big[(\xi-\zeta)
 +{(\xi-\zeta)}^m \big] d\zeta}
 \\
 =\varepsilon \alpha \Big\{\tfrac1{\chi^2}
 \big[e^{\chi \xi}-1-\chi \xi \big]
 +\tfrac{m!}{\chi^{m+1}} \big[e^{\chi \xi}
 -1-\chi \xi-\dots-\tfrac1{m!} {(\chi \xi)}^m
 \big] \Big\} \leq \varepsilon \beta e^{\chi \xi},
\end{multline}
where $\theta_i=\theta_i(\varepsilon \zeta, \varepsilon \xi) \in
(0,1)$, $\|\cdot \|$ is the norm of the space~$C[0,X]$, and
$$
 \alpha:=\|a_{m-1}'(x) \|+\dots+\| a_0'(x) \|, \quad
 \beta :=\alpha\frac{\chi^{m-1} +m!}{\chi^{m+1}}.
$$

For the second integral in~\eqref{Ai} we have (see~\eqref{f} and
\eqref{est wty})
\begin{multline}\label{est int2}
 \int_0^\xi e^{\chi \zeta}
 \big[1+{(\xi-\zeta)}^{m-1} \big]
 \big[|f(\zeta;\varepsilon)|+|\bar y'(\varepsilon \zeta)|
 \big] d\zeta \leq \int_0^\xi e^{\chi \zeta}
 \big[1+{(\xi-\zeta)}^{m-1} \big]
 \\
 \times \Big\{\tilde C
 \big[|a_{m-1}'(\varepsilon \theta_{m-1} \zeta)|+\dots
 +|a_0'(\varepsilon \theta_0 \zeta)| \big]
 (\zeta+\zeta^m) e^{-\chi \zeta}
 +|\bar y'(\varepsilon \zeta)| \Big\} d\zeta
 \\
 \leq \Big\{\tilde C \alpha \max_{\zeta>0}
 \big[(\zeta+\zeta^m) e^{-\chi \zeta} \big]
 +\| \bar y'(x) \| \Big\} \int_0^\xi e^{\chi \zeta}
 \big[1+{(\xi-\zeta)}^{m-1} \big] d\zeta
 \\
 =\Big\{\tilde C \alpha \max_{\zeta>0}
 \big[(\zeta+\zeta^m) e^{-\chi \zeta} \big]
 +\| \bar y'(x) \| \Big\}
 \Big\{\tfrac1\chi \big[e^{\chi \xi}-1 \big]
 \\
 +\tfrac{(m-1)!}{\chi^m}
 \big[e^{\chi \xi}-1-\chi \xi-\dots
 -\tfrac1{(m-1)!} {(\chi \xi)}^{m-1} \big] \Big\}
 \leq \gamma e^{\chi \xi},
\end{multline}
where $\theta_i=\theta_i(\varepsilon \zeta) \in (0,1)$,
\begin{equation*}
 \gamma :=\Big\{\tilde C \alpha \max_{\zeta>0}
 \big[(\zeta+\zeta^m) e^{-\chi \zeta} \big]
 +\| \bar y'(x) \| \Big
 \} \tfrac{\chi^{m-1}+(m-1)!}{\chi^m}.
\end{equation*}

From \eqref{Ai}, \eqref{est int1}, and \eqref{est int2} we see
that if $C_0$ and $\varepsilon$ satisfy the inequalities
\begin{gather}\label{ineq C}
 0 \leq C_0 \varepsilon C_\Phi \beta+C_\Phi \gamma \leq C_0,
\end{gather}
hence $\widehat A(C_0;\varepsilon)[z^1, \dots, z^m](\xi) \in
O(\vartheta,C_0;\varepsilon)$.

We set
\begin{gather}\label{ep0}
 \varepsilon_0 :=\gamma_0 {(C_\Phi \beta)}^{-1},
\end{gather}
where $\gamma_0$ is an arbitrary number from the interval $(0,1)$
(if $\beta=0$, i.e., $a_i(x)=\mathrm{const}$ on $[0,X]$, then
$\varepsilon_0 :=+\infty$) and $C_0 :=C_\Phi
\gamma/(1-\gamma_0)$. Then the inequalities \eqref{ineq C} hold
for any $\varepsilon \in (0,\varepsilon_0]$.
\end{proof}

Assume that for any fixed positive $\varepsilon$ and any
$\varphi_1(\xi)=(z^1_1(\xi), \dots, z^m_1(\xi))$ and
$\varphi_2(\xi)=(z^1_2(\xi), \dots, z^m_2(\xi))$ from
$C_m[0,X/\varepsilon]$, the distance $\rho_\varepsilon$ between
$\varphi_1$ and $\varphi_2$ is defined:
\begin{gather}\label{rho}
 \rho_\varepsilon(\varphi_1, \varphi_2)
 :={\| \varphi_2-\varphi_1 \|}_{C_m[0,X/\varepsilon]}
 :=\max_{\xi \in X(\varepsilon)} \max_{1 \leq i \leq m}
 |z^i_2(\xi)-z^i_1(\xi)|,
\end{gather}
where $X(\varepsilon) :=[0,X/\varepsilon]$. Note that
$C_m[0,X/\varepsilon]$ and $O(\vartheta,C_0;\varepsilon)$ with
$\rho_\varepsilon$ defined above are complete metric spaces.

\begin{sta}
The operator $\widehat A(\varepsilon)$ is a contractive operator
for any $\varepsilon \in (0,\varepsilon_0]$.
\end{sta}

\begin{proof}
Let $\rho_\varepsilon$ be the metric \eqref{rho} of the space
$C_m[0,X/\varepsilon]$. Take two arbitrary functions
$\varphi_1(\xi)=(z^1_1(\xi), \dots, z^m_1(\xi))$ and
$\varphi_2(\xi)=(z^1_2(\xi), \dots, z^m_2(\xi))$ from this space
and, taking into account \eqref{iezi} and~\eqref{estPhi}, estimate
the distance between $\widehat A(\varepsilon)[\varphi_1]$ and
$\widehat A(\varepsilon)[\varphi_2]$:
\begin{multline}\label{metr}
 \rho_\varepsilon\Big(\widehat A(\varepsilon)[\varphi_1],
 \widehat A(\varepsilon)[\varphi_2]\Big)
 =\max_{\xi \in X(\varepsilon)} \max_{1 \leq i \leq m}
 \Big| \widehat A_i(\varepsilon)[\varphi_2](\xi)
 -\widehat A_i(\varepsilon)[\varphi_1](\xi) \Big|
 \\
 =\max_{\xi \in X(\varepsilon)} \max_{1 \leq i \leq m}
 \Big| \int_0^\xi \Phi^m_{\xi^{i-1}}
 (\xi-\zeta; \varepsilon \xi)
 \Big\{ \big[a_{m-1}(\varepsilon \zeta)
 -a_{m-1}(\varepsilon \xi)\big]
 \big[z^m_2(\zeta)-z^m_1(\zeta)\big]+\dots
 \\
 \shoveright{+\big[a_0(\varepsilon \zeta)-a_0(\varepsilon \xi)\big]
 \big[z^1_2(\zeta)-z^1_1(\zeta)\big] \Big\} d\zeta \Big|}
 \\
 \shoveleft{\leq \rho_\varepsilon(\varphi_1, \varphi_2) C_\Phi
 \max_{\xi \in X(\varepsilon)} \int_0^\xi
 e^{\chi (\zeta-\xi)} \big[1+{(\xi-\zeta)}^{m-1} \big]}
 \\
 \times
 \Big[\big|a_{m-1}(\varepsilon \zeta)
 -a_{m-1}(\varepsilon \xi)\big|
 +\dots+\big|a_0(\varepsilon \zeta)
 -a_0(\varepsilon \xi)\big| \Big] d\zeta.
\end{multline}

From \eqref{metr}, \eqref{est int1}, and \eqref{ep0} we conclude
that for any $\varepsilon \in (0,\varepsilon_0]$ the contraction
coefficient $k(\varepsilon)$ of the operator $\widehat
A(\varepsilon)$ satisfies the estimate
\begin{gather}\label{ke}
 k(\varepsilon) \leq \varepsilon C_\Phi \beta
 =\gamma_0 \frac{\varepsilon}{\varepsilon_0} \leq \gamma_0<1.
\end{gather}
The proof is complete.
\end{proof}

Since the contraction coefficient $k(C_0;\varepsilon)$ of the
operator $\widehat A(C_0;\varepsilon)$ certainly does not exceed
$k(\varepsilon)$, the estimate \eqref{ke} is also valid for it:
\begin{gather}\label{kCe}
 k(C_0;\varepsilon) \leq \gamma_0
 \frac{\varepsilon}{\varepsilon_0} \leq \gamma_0<1.
\end{gather}

Thus, we can apply the Banach fixed-point theorem to the
operator~$\widehat A(C_0;\varepsilon)$ and conclude that for any
$\varepsilon \in (0,\varepsilon_0]$ the solution
$(z^1(\xi;\varepsilon), \dots, z^m(\xi;\varepsilon))=:
\varphi(\xi;\varepsilon)$ of the
problem~\eqref{dez1}--\eqref{iczi} (which is equivalent to
Eq.~\eqref{ie+}) belongs to~$O(\vartheta,C_0;\varepsilon)$. We emphasize
that the existence and the global uniqueness (i.e., uniqueness on
the set $[0,X/\varepsilon] \times \RR^m$) of the solution
$\varphi(\xi;\varepsilon)$ (for all $\varepsilon \in \RR$) are
immediately implied by the linearity of the
problem~\eqref{dez1}--\eqref{iczi} (the linearity of
Eq.~\eqref{ie+}).

The contractive property of the operator $\widehat
A(C_0;\varepsilon)$ also allows one to construct the iterative sequence
$\varphi_n(\xi;\varepsilon)=(z^1_n(\xi;\varepsilon), \dots,
z^m_n(\xi;\varepsilon))$ converging with respect to the norm of
the space~$C_m[0,X/\varepsilon]$ to the exact solution
$\varphi(\xi;\varepsilon)$ of the
problem~\eqref{dez1}--\eqref{iczi} uniformly with respect
to~$\varepsilon \in (0,\varepsilon_0]$:
\begin{equation*}
 \big\| \varphi-\varphi_n \big\|_{C_m[0,X/\varepsilon]}
 :=\max_{\xi \in X(\varepsilon)} \max_{1 \leq i \leq m}
 \big|z^i(\xi;\varepsilon)-z^i_n(\xi;\varepsilon)\big|
 \rightarrow 0, \quad n \to \infty.
\end{equation*}

We set $\varphi_0(\xi;\varepsilon) \equiv (0, \dots, 0) =: \vartheta$. Since
$\varphi(\xi;\varepsilon) \in O(\vartheta,C_0;\varepsilon)$, we have
\begin{gather}\label{vph0}
 \big\| \varphi(\xi;\varepsilon)
 -\varphi_0(\xi;\varepsilon) \big\|_{C_m[0,X/\varepsilon]}
 =\big\| \varphi(\xi;\varepsilon) \big\|_{C_m[0,X/\varepsilon]}
 \leq C_0
\end{gather}
for all $\varepsilon \in (0,\varepsilon_0]$.

Further, for any natural $n$ we set
\begin{gather}\label{vph_n}
 \varphi_n(\xi;\varepsilon)
 :=\widehat A(C_0;\varepsilon)
 [\varphi_{n-1}](\xi;\varepsilon).
\end{gather}
Then, taking into account \eqref{kCe} and~\eqref{vph0}, we have
for each $n \in \{0\} \cup \NN=: \NN_0$ and each $\varepsilon \in
(0,\varepsilon_0]$
\begin{multline}\label{est vph_n}
 \big\| \varphi(\xi;\varepsilon)
 -\varphi_n(\xi;\varepsilon) \big\|_{C_m[0,X/\varepsilon]}
 \\
 \leq {k(C_0;\varepsilon)}^n
 \big\| \varphi(\xi;\varepsilon)
 -\varphi_0(\xi;\varepsilon) \big\|_{C_m[0,X/\varepsilon]}
 \leq C_0 \left(\gamma_0 \frac{\varepsilon}{\varepsilon_0}\right)^n.
\end{multline}

We turn to the problem~\eqref{de}--\eqref{ic}. Due to
\eqref{subs}, we obtain the iterative sequences
$y^1_n(x;\varepsilon)$, \dots, $y^m_n(x;\varepsilon)$,
respectively, for the solution $y(x;\varepsilon)$ of the original
problem and its derivatives $\frac{d}{dx} y(x;\varepsilon)$, \dots,
$\frac{d^{m-1}}{{dx}^{m-1}} y(x;\varepsilon)$:
\begin{gather}\label{y^1_n}
 y^1_n(x;\varepsilon)
 :=\tilde y\left(\frac{x}{\varepsilon},x\right)
 +\varepsilon z^1_n\left(\frac{x}{\varepsilon}; \varepsilon\right),
 \quad n \in \NN_0;
 \\
 \label{y^i_n}
 y^i_n(x;\varepsilon) :=\varepsilon^{1-i}
 \Pi^{(i-1)} \left(\frac{x}{\varepsilon}\right)
 +\varepsilon^{2-i}
 z^i_n\left(\frac{x}{\varepsilon}; \varepsilon\right),
 \quad (i,n) \in \overline{2,m} \times \NN_0.
\end{gather}

For $n \geq 1$, the values $y^i_n(x;\varepsilon)$ can be
immediately expressed through $y^i_{n-1}(x;\varepsilon)$:
\begin{align*}
 y^1_n(x;\varepsilon) &=\tilde y\left(\frac{x}{\varepsilon},x\right)
 +\varepsilon \widehat A_1(C_0;\varepsilon)
 [z^1_{n-1}, \dots, z^m_{n-1}]
 \left(\frac{x}{\varepsilon};\varepsilon\right)
 \\
 &=: \widehat B_1(\varepsilon)[y^1_{n-1}, \dots, y^m_{n-1}]
 (x;\varepsilon),
 \\
 y^i_n(x;\varepsilon) &=\varepsilon^{1-i}
 \Pi^{(i-1)}\left(\frac{x}{\varepsilon}\right)
 +\varepsilon^{2-i} \widehat A_i(C_0;\varepsilon)
 [z^1_{n-1}, \dots, z^m_{n-1}]
 \left(\frac{x}{\varepsilon};\varepsilon\right)
 \\
 &=: \widehat B_i(\varepsilon)[y^1_{n-1}, \dots,y^m_{n-1}]
 (x;\varepsilon), \quad i \in \overline{2,m},
\end{align*}
where
\begin{gather*}
 z^1_{n-1}(\xi;\varepsilon)=\varepsilon^{-1}
 \Big[y^1_{n-1}(\varepsilon \xi; \varepsilon)
 -\tilde y(\xi, \varepsilon \xi)\Big],
 \\
 z^i_{n-1}(\xi;\varepsilon)=\varepsilon^{i-2}
 y^i_{n-1}(\varepsilon \xi; \varepsilon)
 -\varepsilon^{-1} \Pi^{(i-1)}(\xi),
 \quad i \in \overline{2,m}
\end{gather*}
(see~\eqref{y^1_n}, \eqref{y^i_n}, and~\eqref{vph_n}) or briefly
\begin{equation*}
 \psi_n(x;\varepsilon) :=\widehat B(\varepsilon)
 [\psi_{n-1}](x;\varepsilon),
\end{equation*}
where $\psi_n(x;\varepsilon) :=(y^1_n(x;\varepsilon), \dots,
y^m_n(x;\varepsilon))$ and $\widehat B(\varepsilon) :=(\widehat
B_1(\varepsilon), \dots \widehat B_m(\varepsilon))$. Note that
the operator $\widehat B(\varepsilon)$ is contractive for
$\varepsilon \in (0,\varepsilon_0]$ (i.e., for the same
$\varepsilon$ as $\widehat A(C_0;\varepsilon)$) and the operator
$\widehat B(\varepsilon)$ satisfies the condition
\begin{equation*}
 \widehat B(\varepsilon) : O(\tilde\psi, C_0; \varepsilon)
 \to O(\tilde\psi, C_0; \varepsilon)
\end{equation*}
for $\varepsilon \in (0,\varepsilon_0]$, where
\begin{gather*}
 O(\tilde\psi, C_0; \varepsilon) := \bigg\{ (y^1, \dots, y^m) \in C_m[0,X] : \forall x \in [0,X]
 \\
 y^1(x) \in \left[\tilde y\left(\frac{x}{\varepsilon},x\right) - \varepsilon C_0, \tilde y\left(\frac{x}{\varepsilon},x\right) + \varepsilon C_0\right],
 \\
 y^2(x) \in \left[\varepsilon^{-1} \Pi'\left(\frac{x}{\varepsilon}\right) - C_0, \varepsilon^{-1} \Pi'\left(\frac{x}{\varepsilon}\right) + C_0\right], \dots,
 \\
 y^m(x) \in \left[ \varepsilon^{1-m} \Pi^{(m-1)}\left(\frac{x}{\varepsilon}\right) - \varepsilon^{2-m} C_0, \varepsilon^{1-m} \Pi^{(m-1)}\left(\frac{x}{\varepsilon}\right) + \varepsilon^{2-m} C_0 \right] \bigg\}
\end{gather*}
is a closed $(\varepsilon C_0, C_0, \dots, \varepsilon^{2-m}
C_0)$-neighborhood of the vector-function
\begin{equation*}
 \tilde\psi\left(\frac{x}{\varepsilon};\varepsilon\right)
 :=\left(\tilde y\left(\frac{x}{\varepsilon},x\right),
 \varepsilon^{-1} \Pi'\left(\frac{x}{\varepsilon}\right), \dots,
 \varepsilon^{1-m} \Pi^{(m-1)}\left(\frac{x}{\varepsilon}\right)
 \right)
\end{equation*}
in the space $C_m[0,X]$.

We estimate the accuracy of the approximation of
$\frac{d^{i-1}}{{dx}^{i-1}} y(x;\varepsilon)$ by
$y^i_n(x;\varepsilon)$. For each $n \in \NN_0$ and $\varepsilon
\in (0,\varepsilon_0]$ we have (see~\eqref{y^1_n}, \eqref{y^i_n},
\eqref{subs}, and~\eqref{est vph_n}):
\begin{gather*}
 \big\| y(x;\varepsilon)-y^1_n(x;\varepsilon) \big\|
 =\left\| y(x;\varepsilon)
 -\tilde y\left(\frac{x}{\varepsilon},x\right)
 -\varepsilon z^1_n\left(\frac{x}{\varepsilon};
 \varepsilon\right) \right\|
 \\
 =\varepsilon \left\| z^1\left(\frac{x}{\varepsilon};
 \varepsilon\right)-z^1_n\left(\frac{x}{\varepsilon};
 \varepsilon\right) \right\|
 \leq \varepsilon
 \left\| \varphi\left(\frac{x}{\varepsilon};\varepsilon\right)
 -\varphi_n\left(\frac{x}{\varepsilon};\varepsilon\right)
 \right\|_{C_m[0,X]}
 \leq C_0 \varepsilon \left(\gamma_0
 \frac{\varepsilon}{\varepsilon_0}\right)^n,
 \\
 \left\| \frac{d^{i-1}}{dx^{i-1}} y(x;\varepsilon)
 -y^i_n(x;\varepsilon) \right\|
 =\left\| \frac{d^{i-1}}{dx^{i-1}} y(x;\varepsilon)
 -\varepsilon^{1-i} \Pi^{(i-1)}\left(\frac{x}{\varepsilon}\right)
 -\varepsilon^{2-i} z^i_n\left(\frac{x}{\varepsilon};
 \varepsilon\right) \right\|
 \\
 =\varepsilon^{2-i} \left\| z^i\left(\frac{x}{\varepsilon};
 \varepsilon\right)-z^i_n\left(\frac{x}{\varepsilon};
 \varepsilon\right) \right\|
 \leq \varepsilon^{2-i} \left\|
 \varphi\left(\frac{x}{\varepsilon};\varepsilon\right)
 -\varphi_n\left(\frac{x}{\varepsilon};\varepsilon\right)
 \right\|_{C_m[0,X]}
 \\
 \leq C_0 \varepsilon^{2-i}
 \left(\gamma_0 \frac{\varepsilon}{\varepsilon_0}\right)^n,
 \quad i \in \overline{2,m}.
\end{gather*}

\bibliographystyle{utphys}

\bibliography{Singul}

\end{document}